\documentclass[a4paper,11pt]{article}

\usepackage{amssymb}
\usepackage{amsmath}
\usepackage{amsthm}
\usepackage{graphicx}
\usepackage{subcaption}
\usepackage[color=white]{todonotes}
\usepackage{float}
\usepackage{enumerate}
\usepackage{color}
\usepackage{hyperref}

\usepackage{geometry}
\newtheorem{definition}{Definition}[section]
\newtheorem{theorem}[definition]{Theorem}
\newtheorem{proposition}[definition]{Proposition}
\newtheorem{lemma}[definition]{Lemma}
\theoremstyle{remark}
\newtheorem{remark}[definition]{Remark}

\numberwithin{equation}{section}

\usepackage{graphicx,epstopdf}

\DeclareMathOperator{\diam}{diam}

\newcommand{\f}{\mathrm{f}}

\newcommand{\intd}{\,\mathrm{d}}
\newcommand{\R}{\mathbb{R}}

\newcommand{\PLOD}{P_\mathrm{LOD}}
\newcommand{\VLOD}{V_\mathrm{LOD}}
\newcommand{\uLOD}{u_\mathrm{LOD}}

\newcommand{\LOD}{\mathrm{LOD}}
\newcommand{\lambdaLOD}{\lambda_\mathrm{LOD}}

\newcommand{\quotes}[1]{``#1''}

\begin{document}

\begin{center}
{\LARGE On optimal convergence rates for discrete minimizers of the Gross-Pitaevskii energy in LOD spaces
 \hspace{-3pt}\footnote{The authors acknowledge the support by the G\"oran Gustafsson foundation.}}\\[2em]
\end{center}

\begin{center}
{\large Patrick Henning\footnote[1]{Department of Mathematics, Ruhr University Bochum, DE-44801 Bochum, Germany, \\ e-mail: \textcolor{blue}{patrick.henning@rub.de}.} and 
Anna Persson\footnote[2]{Department of Mathematics, KTH Royal Institute of Technology, SE-100 44 Stockholm, Sweden., \\ e-mail: \textcolor{blue}{apersson@it.uu.se}.}}\\[2em]
\end{center}

\begin{abstract}
In this paper we revisit a two-level discretization based on the Localized Orthogonal Decomposition (LOD). It was originally proposed in [P.Henning, A.M{\aa}lqvist, D.Peterseim. SIAM J. Numer. Anal. 52(4):1525--1550, 2014] to compute ground states of Bose-Einstein condensates by finding discrete minimizers of the Gross--Pitaevskii energy functional. The established convergence rates for the method appeared however suboptimal compared to numerical observations and a proof of optimal rates in this setting remained open. In this paper we shall close this gap by proving optimal order error estimates for the $L^2$- and $H^1$-error between the exact ground state and discrete minimizers, as well as error estimates for the ground state energy and the ground state eigenvalue. In particular, the achieved convergence rates for the energy and the eigenvalue are of $6$th order with respect to the mesh size on which the discrete LOD space is based, without making any additional regularity assumptions. These high rates justify the use of very coarse meshes, which significantly reduces the computational effort for finding accurate approximations of ground states. In addition, we include numerical experiments that confirm the optimality of the new theoretical convergence rates, for both smooth and discontinuous potentials.
\end{abstract}

\section{Introduction}
In this paper we consider the stationary Gross-Pitaevskii equation (GPE), which can be written as an eigenvalue problem with eigenvector nonlinearity seeking $u\in H^1_0(\Omega)$ with $\| u\|_{L^2(\Omega)}=1$ and $\lambda \in \R$ such that
\begin{align*}
-\Delta u + Vu + \beta|u|^2u &= \lambda u.
\end{align*}
Here, $\Omega \subset \R^d$ (with $d=1,2,3$) is a computational domain, $V(x)$ models a trapping potential and $\beta \in \R_{\ge 0}$ is a constant. The perhaps most prominent application of the GPE is the modeling of so-called Bose--Einstein condensates \cite{Bos24,Ein24,Gro61,Pit61}. Bose--Einstein condensates (BECs) are formed when a dilute gas of bosons is cooled to ultra-low temperatures so that almost all particles occupy the same quantum state, i.e. they become indistinguishable from each other and behave like one giant \quotes{macro particle}. The first experimental realization of a BEC goes back to the Nobel prize-winning works by Cornell, Ketterle and Wieman \cite{AEM95,DMA95}. Bose--Einstein condensates are of great interest in the study of macroscopic quantum effects such as superfluidity \cite{ARV01,MAH99} and the closely related phenomenon of superconductivity \cite{Hashimoto20}.

In the context of BECs, $\beta \ge 0$ models a repulsion parameter that depends on the number of bosons, as well as their type and mass. In this setting, the GPE has a positive spectrum, where the smallest eigenvalue is simple. This eigenvalue $\lambda$ is called the ground state eigenvalue of the problem and it describes the chemical potential of a BEC in the lowest energy state. The corresponding ground state eigenfunction $u$ is unique up to its sign. An equivalent characterization of the ground state $u$ is to find a minimizer of an energy functional $E$ (defined in equation \eqref{definition-energy-functional} below) over the space of functions in $H^1_0(\Omega)$ with the normalization constraint $\| u \|_{L^2(\Omega)}=1$.

Due to the complexity of physical experiments, computing the ground sate numerically is a very important task. Starting from the energy minimizing perspective, a numerical method has two essential components. The first component is an iterative method that allows to find minimizers of the energy $E$ on the constrained manifold. Here there exists a variety of well-established schemes, where we exemplarily refer to methods based on self consistent field iterations (SCF) \cite{CaL00,CaL00B,DiC07,UJR21}, methods based on discrete Sobolev gradient flows / Riemannian optimization \cite{ALT17,AnD14,APS21,BD04,BWM05,DaK10,DaP17,HeP20,HSW21,KaE10,Zha21} or the \quotes{$J$-method} \cite{JKM14,AHP21b}. We shall however not further discuss this aspect in our paper. Instead, we shall focus on the second important component of any numerical method for the computation of ground states: the choice of a suitable discrete space in which the energy functional is minimized. Standard choices are Lagrange finite element spaces (which can be often associated with finite difference discretizations) or Fourier spaces. These standard choices were analytically explored in  \cite{CCM10}. In these discrete spaces, each iteration of a scheme for finding a discrete minimizer of $E$ requires the solution of a (typically elliptic) problem. Hence, the complexity of the final method depends naturally on the choice and the dimension of the discrete spaces.

In addition to the aforementioned standard choices, the usage of an alternative discrete space was proposed in \cite{HMP14}. Besides achieving high convergence orders, the proposed spaces have the advantage that they achieve their full potential already under minimal regularity assumptions on $u$ and the potential $V$. Their construction is based on techniques from numerical homogenization, which build a low-dimensional generalized finite element space using a localized orthogonal decomposition (LOD). The LOD splits an ideal solution space into a low-dimensional approximation space and a high-dimensional remainder space. The low-dimensional space (to which we refer to as the LOD-space) is enriched with problem-specific features and admits a set of quasi-local basis functions that can be efficiently used in practical computations. In particular, due to the very good approximation properties and the quasi-locality of a basis, LOD-spaces can be easily used in a traditional Galerkin approach. Historically, the LOD was first introduced in \cite{MP14} for elliptic multiscale problems with a rough coefficient. Since then, it has been further developed to suit a range of different problems, e.g., parabolic equations \cite{MP18,MP17,MPS18}, equations describing wave phenomena, see, e.g., \cite{AH17, P17, HP20, GHV18, LMP21}, non-linear problems \cite{HMP14_b,V21}, and more. In the context of nonlinear Schr\"odinger equations and GPEs, LOD techniques have been suggested in \cite{HMP14,HeW21}. For a review of the LOD we refer to the textbook by M{\aa}lqvist and Peterseim \cite{MaP21book} and the recent survey article on numerical homogenization \cite{AHP21}.

In this paper we revisit the LOD for the computation of ground states of the Gross-Pitaevskii equation (as introduced in \cite{HMP14}) in order to prove the higher order convergence rates that were numerically observed by the authors, but which have not yet been theoretically established. To be precise, the a priori error estimates proved in \cite{HMP14} for the GPE predict convergence rates of order $O(H^2)$ for the $H^1$-error, of order $O(H^3)$ for the $L^2$-error and the eigenvalue error and of order $O(H^4)$ for the error in energy, where $H$ is the \quotes{coarse} mesh size of the LOD-space. As already noted in \cite{HMP14}, these rates seemed suboptimal when compared to numerical experiments. In fact, in this contribution we prove the optimal rates, which are of order $O(H^3)$ for the $H^1$-error, of order $O(H^4)$ for the $L^2$-error and (with a dramatic improvement) of order $O(H^6)$ for the eigenvalue-error and the error in energy. For that we will not require any additional regularity assumptions and the results still hold for rough potentials $V \in L^{\infty}(\Omega)$, where the $H^1$- and energy estimates even hold for $V \in L^{2+\sigma}(\Omega)$ with any $\sigma>0$.

Besides the obvious advantage of efficiently approximating ground states in the LOD space, our new results have direct implications for the simulation of the dynamics of Bose--Einstein condensates. In such a setting, the dynamics are described by nonlinear Schr\"odinger equations (NLS), where the initial values are typically ground states of the GPE (with respect to some modified configuration). In \cite{HeW21} it was recently proved that when the NLS is discretized with LOD spaces, then the energy is approximated and conserved with an accuracy of order $O(H^6)$, provided that this accuracy can be already guaranteed for the initial value. Hence, this is exactly what we are establishing in this paper, as it justifies that the ground states computed in LOD-spaces can be straightforwardly used as initial values in an NLS without sacrificing the accuracy with which the energy is approximated over time.

Finally we note that in \cite{HMP14}, an additional post-processing step on a finer mesh is suggested, which improves the quality of the approximation even further. We shall not focus on the post-processing in this paper. However, we note that the convergence orders can be further improved by such a technique.

The outline of the paper is as follows; in Section~\ref{sec:GPE} the Gross-Pitaevskii equation is presented in more detail, in Section~\ref{sec:discretization} we describe the LOD discretization technique, in Section~\ref{sec:superconvergence} we present proofs of the optimal convergence rates in the $H^1$- and $L^2$-norm as well as for the energy and eigenvalue. Finally, in Section~\ref{sec:experiments} we provide two numerical experiments that confirm the predicted convergence rates, for both smooth and discontinuous potentials.
 
\section{The Gross-Pitaevskii eigenvalue problem}\label{sec:GPE}
Let $\Omega \subset \R^d$ be a bounded Lipschitz domain for $d=1,2,3$, with $\| v \| := \| v \|_{L^2(\Omega)}$ being the $L^2$-norm and $(v,w):=(v,w)_{L^2(\Omega)}$ being the $L^2$-inner product on $\Omega$. Furthermore, we denote by $H^1_0(\Omega)$ the Sobolev space of $L^2$-integrable and weakly-differentiable functions with a vanishing trace on the boundary $\partial \Omega$. With this, $\| v \|^2_{H^1(\Omega)}:=\| v\|^2 + \| \nabla v\|^2$ is the standard norm on $H^1(\Omega)$ and we denote by $\langle F , v \rangle:=\langle F , v \rangle_{H^{-1}(\Omega),H^1_0(\Omega)}$ the canonical duality pairing on $H^1_0(\Omega)$ and its dual space.

We consider the Gross--Pitaevskii energy functional $E: H^1_0(\Omega) \to \R$ defined by
\begin{align}
\label{definition-energy-functional}
E(\phi) := \frac{1}{2}\int_{\Omega} |\nabla\phi|^2 \intd x + \frac{1}{2} \int_{\Omega} V |\phi|^2 \intd x + \frac{1}{4}\int_{\Omega} \beta |\phi|^4 \intd x.
\end{align}
Here we make the following assumptions on the domain $\Omega$, the trapping potential $V$ and the interaction parameter $\beta$:
\begin{enumerate}[({A}1)]
	\item $\Omega$ is a convex domain with polygonal boundary (for $d=1,2,3$).
	\item $V \in L^{2+\sigma}(\Omega)$ and non-negative, where $\sigma>0$ for $d=3$ and $\sigma=0$ for $d=1,2$.
	\item $\beta \in \R_{\ge 0}$. 
\end{enumerate}
Observe that $E$ is two times Fr\'echet-differentiable, where we denote the corresponding first and second derivatives by $E^{\prime}$ and $E^{\prime\prime}$ respectively.

Under the above assumptions, we can seek the ground state of the energy $E$. The ground state is defined as a function $u\in H^1_0(\Omega)$ that minimizes $E(\cdot)$ under the normalization constraint that $\int_{\Omega} |u|^2  \intd x =1$, i.e.
\begin{align}\label{weak_problem}
E(u) = \inf_{\substack{\phi \in H^1_0(\Omega)\\\|\phi\|=1}} E(\phi).
\end{align}
The ground state is unique up to a sign and it can be equivalently characterized by an eigenvalue problem with eigenvector nonlinearity. The following result summarizes these analytical properties and a corresponding proof can be found in \cite{CCM10}.

\begin{theorem}\label{theorem-analytical-properties-u}
Under the assumptions (A1)-(A3), there exists a unique global minimizer of problem \eqref{weak_problem} with the property that $u$ is H\"older-continuous on $\overline{\Omega}$ and with $u>0$ on $\Omega$. \emph{The only other global minimizer is $-u$.}

Furthermore, the unique positive minimizer $u$ can be equivalently characterized by the Gross--Pitaevskii eigenvalue problem (GPE), which seeks the nonnegative eigenfunction $u\in H^1_0(\Omega)$ with $\| u \|=1$ and corresponding smallest eigenvalue $\lambda>0$, such that 
\begin{align}
-\Delta u + Vu + \beta|u|^2u &= \lambda u, \qquad \text{in } \Omega.\label{problem_1}
\end{align}
The smallest eigenvalue $\lambda$ and the ground state energy $E(u)$ are connected through
\begin{align}
\label{groundstate-EV-from-energy}
 \lambda = 2E(u) + \frac{\beta}{2}\|u\|^4_{L^4(\Omega)}.
\end{align}
\end{theorem}
Note the nontrivial observation in Theorem \ref{theorem-analytical-properties-u} that if $\lambda$ is the smallest eigenvalue of problem \eqref{problem_1}, then the corresponding eigenfunction $u$ is a ground state of $E$. This means that minimizing the energy and minimizing the eigenvalue is equivalent.

We also note that \eqref{weak_problem} is a weak formulation that can be written using the Fr\'echet derivative of $E$. More precisely, with
\begin{align*}
\langle E^{\prime}(u) , v \rangle = (\nabla u, \nabla v) + (Vu,v) + (\beta|u|^2u, v),
\end{align*}
we observe that any eigenpair $(\lambda,u)$ of \eqref{problem_1} satisfies 
\begin{align}
\label{problem_1_weak}
\langle E^{\prime}(u) , v \rangle  &= \lambda \, ( u,v) \qquad \mbox{for all } v \in H^1_0(\Omega).
\end{align}
This formulation expresses that the eigenvalue problem should be seen as finding the critical points of the (even, positive and convex) functional $E$ on the $L^2$-sphere 
$$\mathbb{S}:=\{ \phi \in H^1_0(\Omega)| \hspace{2pt} \| \phi \| = 1 \}.
$$ Consequently, we also know that the spectrum of the differential operator is real and unbounded.

In the more general form of the Gross--Pitaevskii eigenvalue problem, where a rotational term is included, the eigenfunctions $u$ are typically complex-valued, but the spectrum remains real. We write $|u|^2u$ instead of $u^3$ (which is the same for real functions), to work with the standard notation for the Gross--Pitaevskii equation that remains valid also in complex-valued settings.

Using the equivalence of the energy minimization perspective \eqref{weak_problem} and the GPE \eqref{problem_1} (respectively \eqref{problem_1_weak}), we can easily prove additional regularity for the ground state $u$ that will be useful in the convergence analysis.
\begin{lemma}\label{lemma_bounds_u}
	Assume (A1)-(A3). Let $u$ be a minimizer of \eqref{weak_problem} and let $\lambda >0$ denote the corresponding ground state eigenvalue in the sense of \eqref{problem_1}, then $u \in C^{0}(\overline{\Omega}) \cap H^2(\Omega)$ and $ |u|^2u \in H^1_0(\Omega) \cap H^2(\Omega)$ and we have the following bounds
	\begin{eqnarray}
	\|u\|_{H^2(\Omega)} &\leq& C \hspace{2pt}( \lambda + \| V\| \, \| u \|_{L^{\infty}(\Omega)} + \beta\|u\|^3_{H^1(\Omega)}), \label{elliptic_reg}\\
	\| |u|^2u\|_{H^2(\Omega)} &\leq& C \|u\|^3_{H^2(\Omega)}, \label{nonlinear_reg}
	\end{eqnarray}
where $C$ is a generic constant that only depends on $\Omega$ and $d$. Note that $\|u\|_{H^1(\Omega)}$ can be bounded again by $\sqrt{\lambda}$.
\end{lemma}
\begin{proof}
    In the following we write $a \lesssim b$ to abbreviate $a \le C\, b$ for a constant $C$ that only depends on $\Omega$ and $d$.
	In the first step we write
	\begin{align}
	\label{test-eqn-g}
	-\Delta u = \lambda u -Vu - \beta|u|^2u =: g.
	\end{align}
	If $g \in L^{q/2}(\Omega)$, for $q>d$, it follows that $u \in C^0(\bar\Omega)$, see e.g. \cite[Theorem 8.30]{GiT01}. To verify that this condition is fulfilled, we define $q$ such that $6q/(12 - q) = 2 +\sigma$, where $\sigma$ is as in (A2). For $d=1$ and $d=2$, this gives us $q=3$ and consequently also $q>d$ as desired. For $d=3$, the condition $q>d$ is equivalent to
	\begin{align*}
	2 + \sigma = \frac{6}{\frac{12}{q} - 1}  &> \frac{6}{\frac{12}{3} - 1} = 2, 
	\end{align*}
	which is consistent with the assumption $\sigma>0$ in (A2). Without loss of generality, we further assume $\sigma \le 1$ so that $q\le 4$ for $d=3$.
	
	For $q$ as introduced above, we need to prove now that $g \in L^{q/2}(\Omega)$. For that we use H\"{o}lder's inequality and the Sobolev embedding $H^1_0(\Omega) \hookrightarrow L^6(\Omega)$ ($d \leq 3$) to achieve
	\begin{eqnarray*}
	\|g\|_{L^{q/2}(\Omega)} 
	&\lesssim& \lambda \| u\|_{L^{q/2}(\Omega)}  + \|V\|_{L^{6/(\frac{12}{q} - 1)}(\Omega)}\|u\|_{L^{6}(\Omega)} + \|\beta|u|^2u\|_{L^{q/2}(\Omega)} \\
	&\overset{q\le 4}{\lesssim}& \lambda + \|V\|_{L^{2+\sigma}(\Omega)} \|u\|_{H^1(\Omega)} + \|\beta|u|^2u\|,
	\end{eqnarray*}
	where we exploited $\|u\|=1$. The term $ \|V\|_{L^{2+\sigma}(\Omega)}$ is bounded by (A2) and for the last term we get
	\begin{align}\label{der_order_zero}
	\|\beta|u|^2u\| = \beta \|u\|^3_{L^{6}(\Omega)} \lesssim \beta\|u\|^3_{H^1(\Omega)}.
	\end{align}
	 Consequently, we verified  $g \in L^{q/2}(\Omega)$ for $q>d$ and we have $u \in C^{0}(\overline{\Omega})$. It is worth to mention that $u \in C^0(\overline{\Omega})$ was already proved in \cite{CCM10}, however, we could not verify the result under the stated regularity assumptions which slightly differ from ours. Therefore we decided to explicitly elaborate the argument here.
	 
	 Now that we have $u \in C^{0}(\overline{\Omega})$, we can use this result to show that we indeed even have $g \in L^2(\Omega)$, as
	\begin{align*}
	\| g \| \lesssim \lambda + \| V\| \, \| u\|_{L^{\infty}(\Omega)} +  \beta\|u\|^3_{H^1(\Omega)}.
	\end{align*}
	Using standard elliptic regularity theory for equation \eqref{test-eqn-g} we conclude that $u \in H^2(\Omega)$ with the estimate
	\begin{align*}
	\| u \|_{H^2(\Omega)} \lesssim \| \triangle u \| = \| g \|  \lesssim \lambda + \| V\| \, \| u\|_{L^{\infty}(\Omega)} +  \beta\|u\|^3_{H^1(\Omega)}.
	\end{align*}
	Hence, the bound in \eqref{elliptic_reg} follows.
	
	For \eqref{nonlinear_reg} note that
	\begin{align*}
	\||u|^2u\|^2_{H^2(\Omega)} = \|u^3\|^2_{H^2(\Omega)} = \sum_{|\alpha|\leq 2} \|D^\alpha u^3\|^2,
	\end{align*}
	where the multi-index $\alpha$ denotes the partial derivatives. For $|\alpha|=0$ the term reduces to \eqref{der_order_zero}.  For the first order derivatives, i.e., for $|\alpha|=1$, we get
	\begin{align}\label{der_order_one}
	\|D^\alpha u^3\| \leq \|3u^2D^\alpha u\|,
	\end{align}
	where H\"{o}lder's inequality and the Sobolev embedding gives
	\begin{align*}
	\|u^2D^\alpha u\| &
	\leq \|u\|^2_{L^6(\Omega)} \|D^\alpha u\|_{L^6(\Omega)} \lesssim \|u\|^2_{H^1(\Omega)}\|D^\alpha u\|_{H^1(\Omega)} \lesssim \|u\|^2_{H^1(\Omega)}\|u\|_{H^2(\Omega)}.
	\end{align*}
	For the second order derivatives, i.e. for $|\alpha|=2$, we have
	\begin{align}\label{der_order_two}
	\|D^\alpha u^3\| &\leq \|6uD^{\alpha_1}u D^{\alpha_2}u\| + \|3 u^2D^\alpha u\|,
	\end{align}
	where $\alpha_1,\alpha_2$ are such that $\alpha = \alpha_1 + \alpha_2$.
	For the first term \eqref{der_order_two} we can use H\"{o}lder's inequality and the Sobolev embedding once again to obtain
	\begin{align*}
	\|uD^{\alpha_1}u D^{\alpha_2}u\| &\leq \|u\|_{L^6(\Omega)}\|D^{\alpha_1}u\|_{L^6(\Omega)}\|D^{\alpha_2}u\|_{L^6(\Omega)} \\&\lesssim \|u\|_{H^1(\Omega)}\|D^{\alpha_1}u\|_{H^1(\Omega)}\|D^{\alpha_2}u\|_{H^1(\Omega)} \lesssim \|u\|_{H^1(\Omega)}\|u\|^2_{H^2(\Omega)}.
	\end{align*}
	For the second term in \eqref{der_order_two} we get
	\begin{align*}
	\|u^2D^\alpha u\| &\leq \|u^2\|_{L^\infty(\Omega)}\|D^\alpha u\|.
	\end{align*}
	Using Sobolev's inequality we deduce $\|u^2\|_{L^\infty(\Omega)} = \|u\|^2_{L^\infty(\Omega)} \lesssim C\|u\|^2_{H^2(\Omega)}$.
	This gives 
	\begin{align*}
	\|u^2D^\alpha u\| \leq C\|u\|^3_{H^2(\Omega)}.
	\end{align*}
	Combining all the terms and using $\|u\|_{H^1(\Omega)} \leq \|u\|_{H^2(\Omega)}$ we achieve \eqref{nonlinear_reg}.
\end{proof}

\section{Discretization}\label{sec:discretization}
Approximations of the ground state can be found by minimizing the energy $E$ over a discrete space $\widehat{V}_{H}$ subject to the $L^2$-normalization constraint. The following abstract approximation result was established in \cite[Theorem 1 and Remark 2]{CCM10} and will be a central ingredient for our proofs.
\begin{theorem}
\label{theorem-approximation-properties-ground-states}
Assume (A1)-(A3) and let $u$ be the ground state of \eqref{weak_problem} (up to sign). Furthermore, let $(\widehat{V}_{H})_{{H}>0}$ be a family of finite-dimensional subspaces of $H^1_0(\Omega)$ with the property that
\begin{align*}
\min_{v_{H} \in \widehat{V}_{H} } \| u - v_H \|_{H^1(\Omega)} \overset{H \rightarrow 0}{\longrightarrow} 0.
\end{align*}
For each space $\widehat{V}_{H}$, let $\hat{u}_{H} \in \widehat{V}_{H}$ with $\| \hat{u}_{H} \|=1$ (i.e. $\hat{u}_{H} \in \mathbb{S}$) denote a discrete ground state, that is
\begin{align*}
E(\hat{u}_{H}) =\underset{v_{H} \in \widehat{V}_{H} \cap \mathbb{S}}{\mbox{\rm inf}}\hspace{2pt} E(v_{H}),
\end{align*}
with the property $(u , \hat{u}_{H})\ge 0$. Then there exist (problem-dependent) generic constants $C,c_1,c_2>0$ such that
\begin{align*}
\| u - \hat{u}_{H} \|_{H^1(\Omega)} \le C \min_{v_{H} \in \widehat{V}_{H} } \| u - v_{H} \|_{H^1(\Omega)},
\end{align*}
i.e., $\hat{u}_{H}$ is a quasi-best approximation to $u$ in the $H^1(\Omega)$-norm; and it also holds
\begin{align}
\label{estimate-nergy-abstract}
c_1 \| u -  \hat{u}_{H} \|_{H^1(\Omega)}^2 \le E(\hat{u}_{H}) - E(u)  \le c_2 \| u -  \hat{u}_{H} \|_{H^1(\Omega)}^2.
\end{align}
Finally, for the $L^2$-error it holds
\begin{align*}
\|u - \hat{u}_{H}  \|^2 \leq C \| u- \hat{u}_{H} \|_{H^1(\Omega)} \inf_{ v_{H} \in \widehat{V}_{H}} \| \psi_{u-\hat{u}_{H}} - v_H \|_{H^1(\Omega)},
\end{align*}
where for $\omega \in H^{-1}(\Omega)$, the function
\begin{align*}
\psi_\omega \in V^\perp_u := \{v\in H^1_0(\Omega) | \hspace{2pt} (u,v) = 0 \}
\end{align*}
is the unique solution to the dual problem
\begin{eqnarray}\label{dual_problem_ext_delta}
\langle (E^{\prime\prime}(u) -\lambda) \psi_\omega,  v \rangle = \langle \omega , v \rangle \qquad\mbox{for all } v \in V^{\perp}_u.
\end{eqnarray}
Here, $\lambda>0$ denotes the ground state eigenvalue given by \eqref{groundstate-EV-from-energy} and $E^{\prime\prime}(u) -\lambda$ is a linear elliptic operator with $\langle (E^{\prime\prime}(u) -\lambda) v,  v \rangle \ge C_{E} \| v\|_{H^1(\Omega)}^2$ for some fixed constant $C_E>0$ and for all $v\in H^1_0(\Omega)$.
\end{theorem}
It is easily checked that $E^{\prime\prime}(u)$ can be computed as
\begin{align}
\label{identity-Eprimeprime-prev}
\langle E^{\prime\prime}(u) v , w \rangle = (\nabla v, \nabla w) + (V\,v,w) + 3 (\beta|u|^2 v , w) \quad \mbox{for } v,w \in H^1_0(\Omega).
\end{align}
Note that with \eqref{problem_1_weak} we obtain
\begin{align}
\label{identity-Eprimeprime}
\langle E^{\prime\prime}(u) u , v \rangle = \langle E^{\prime}(u) , v \rangle + 2 (\beta|u|^2 u , v ) = \lambda\, (u,v) +  2 (\beta|u|^2 u , v ).
\end{align}
The results established in \cite{CCM10} also contain an abstract approximation result for the eigenvalue, which is however not optimal for all discrete spaces and needs to be therefore revisited individually depending on the considered discrete setting.

Following \cite{HMP14}, our next goal is the construction of a suitable discrete space that allows for very high convergence orders under minimal regularity assumptions. For that we start from a conventional conforming finite element discretization based on a (coarse) quasi-uniform triangulation $\mathcal{T}_H$ of the domain $\Omega$. Here $H$ denotes the mesh size parameter, i.e. $H =\max\limits_{T \in \mathcal{T}_H} \diam(T)$. With this, let $V_H$ be a classical $P1$ finite element space on $\mathcal{T}_H$ with
\begin{align*}
V_H &:= \{v \in C(\overline \Omega) \cap H^1_0(\Omega)|\hspace{4pt} v_{\vert_T} \text{ is a polynomial of degree $\leq$ 1}, \forall T \in \mathcal{T}_H \}.
\end{align*}
The classical finite element method seeks an approximation $u_H\in V_H$ with $\|u_H\|=1$, and
\begin{align}\label{fem_classic}
E(u_H) = \underset{v_H \in V_H \cap \mathbb{S}}{\mbox{\rm inf}}\hspace{2pt} E(v_H),
\end{align}
where the corresponding eigenvalue is given by $\lambda_H = 2E(u_H) + \tfrac{\beta}{2}\|u_H\|^4_{L^4(\Omega)}$.

Since $E$ is weakly lower semi-continuous on $H^1_0(\Omega)$ and bounded from below by zero, discrete minimizers always exist on finite dimensional spaces. However, it is typically not clear if the solution to \eqref{fem_classic} is also unique up to sign, see e.g. the discussion in \cite{CCM10}. Independent of that uniqueness, it can be shown that for any minimizer $u_H$ with $(u_H,u) \ge 0$ (i.e. consistent with the sign of $u$) it holds
\begin{align*}
\| u_H - u \|_{H^1(\Omega)} &\le C H, \hspace{40pt} \| u_H - u \| \le C H^2, \\
| \lambda_H - \lambda | &\le C H^2 \quad  \mbox{and}\quad E(u_H) - E(u) \le C H^2,
\end{align*}
where $u$ is the unique (up to sign) ground state and $C$ is a constant independent of $H$. The space that we construct in the next subsection will have the same dimension as $V_H$, but it will be able to boost the convergence rates to 3rd order for the $H^1$-error, 4th order for the $L^2$-error and 6th order for eigenvalue and energy error.

\subsection{Localized Orthogonal Decomposition}
In the following we summarize the technique proposed in \cite{HMP14} to make the paper self-contained. Let $P_H:H^1_0(\Omega)\to V_H$ denote the classical $L^2$-projection onto $V_H$ and let 
$$V_\f:=\ker (P_H) = \{v \in H^1_0(\Omega) | \hspace{4pt} P_H(v) = 0\},$$ 
be the kernel of the projection. Due to $H^1$-stability of the $L^2$-projection on quasi-uniform meshes (cf. \cite{BaY14}), the space $V_\f$ is a closed subspace of $H^1_0(\Omega)$. We expect $V_\f$ to contain fine scale details from $H^1_0(\Omega)$ that are not captured by the coarse space $V_H$. By definition we obtain the following $L^2$-orthogonal splitting 
\begin{align*}
H^1_0(\Omega) = V_H \oplus V_\f,
\end{align*}
meaning that any $v \in H^1_0(\Omega)$ can be uniquely decomposed as $v = v_H + v_\f$ such that $v_H \in V_H$ and $v_\f \in V_\f$. The decomposition is also orthogonal meaning that
\begin{align*}
(v_H,v_\f) = 0.
\end{align*}
Next, we define an inner product $a(\cdot,\cdot)$ on $H^1_0(\Omega)$, that is based on the linear terms in the GPE \eqref{problem_1}, by
\begin{align*}
a(v,w) := (\nabla v, \nabla w) + (V \, v, w ).
\end{align*}
This inner product induces another orthogonal splitting, namely
\begin{align*}
H^1_0(\Omega) = \VLOD \oplus V_\f, 
\end{align*}
where 
$$
\VLOD:=\{v\in H^1_0(\Omega)| \hspace{4pt} a(v,w) = 0 \quad \mbox{for all } w\in V_\f \}
$$
is the orthogonal complement to $V_\f$ with respect to $a(\cdot,\cdot)$. Note that, by construction, the space $\VLOD$ has the same dimension as $V_H$.  This is precisely the space in which we will look for discrete minimizers. This is fixed in the following definition.

\begin{definition}[LOD ground state approximation]
\label{def-LOD-groundstates}
The LOD ground state approximation is obtained by using the $\VLOD$ space in \eqref{weak_problem}. That is, we seek $\uLOD \in \VLOD$ such that $(\uLOD,u) \ge 0$, $\|\uLOD\|=1$, and
\begin{align}\label{LOD_approx}
E(\uLOD) = \inf_{ v \in \VLOD \cap \mathbb{S} } E(v),
\end{align}
where the corresponding discrete eigenvalue is given by 
$$
\lambdaLOD := 2 E(\uLOD) + \tfrac{\beta}{2}\|\uLOD\|^4_{L^4(\Omega)}.
$$
\end{definition}

\begin{remark}[Practical realization]
\label{remark-practical-realization}
The practical realization of \eqref{LOD_approx} requires some remarks. First of all, the condition $(\uLOD,u) \ge 0$ is never enforced in the numerical method, as it does not matter if $\uLOD$ converges to the positive or the negative ground state, the respective approximation properties are identical.

Second, for the construction of $\VLOD$ it is necessary to solve a set of saddle point problems that determine the  basis functions of $\VLOD$. For an efficient realization, this includes a localization step which is analytically justified by an exponential decay of the basis function in units of $H$ (cf. \cite{HeP13,MP14}). Furthermore, the space $V_\f$ needs a discrete representation on a finer mesh of size $h<H$. Comprehensive details on all these implementation details are given in \cite{ENGWER2019123}.

The influence of these additional approximations (i.e. truncation/localization and fine scale representation of the LOD basis functions) has been studied extensively in the literature and we note that they do not lead to a reduction of the convergence orders of the final method when chosen appropriately. For the Gross--Pitaevskii eigenvalue problem, the influence has been analytically studied in \cite{HMP14}. 

Further practical tricks for an efficient treatment of the nonlinear term are discussed in \cite{HeW21}.

Finally we also stress that aside from the construction of the LOD space, a practical realization also requires an iterative method for finding a discrete minimizer.  Here various techniques are possible, where we refer to the introduction for a corresponding literature survey.

As seen by the above discussion, the total computational complexity of the method depends on various choices (such as the degree of localization and the selected iterative solver). In some cases, also the structure of the potential $V$ can be exploited to actually only compute  a few LOD basis functions and obtain the remaining ones by rotations and translations. This leads to significant computational savings. Additional speedups can be obtained by using inexact solves for the iterative solver. Due to the variety of choices, general statements about the computational complexity are barely possible. However, we stress that the sparsity structure and the condition number of system matrices involving the LOD space are of the same order as for system matrices in the standard coarse space $V_H$. Any overhead is purely caused by the precomputation of LOD basis functions and their storage.
\end{remark}
In the light of Remark \ref{remark-practical-realization}, all the proofs below are based on the ideal setting of Definition \ref{def-LOD-groundstates}, i.e., the setting without localization and fine scale representation. This simplifies both the arguments and the notation. By considering localization, an additional term depending on the size of the localization patches would appear in the error estimates. The implications of truncation and fine scale discretization are already thoroughly discussed in \cite{HMP14} for the GPE, which is why we omit this part in our paper. This allows us to keep the presentation short and to focus on the main novelty of this contribution, namely the arguments that actually improve the convergence orders in the proofs.

We have the following main result, which we shall prove in Section \ref{sec:superconvergence}.

\begin{theorem}\label{thm_main}
	Assume (A1)-(A3) and let $u$ and $\uLOD$ be solutions to \eqref{weak_problem} and \eqref{LOD_approx}, respectively. Then it holds
\begin{align*}
\nonumber\| \uLOD - u \|_{H^1(\Omega)} \le C H^3 \qquad \mbox{and} \qquad E(\uLOD) - E(u) \le C H^6.
\end{align*}
If furthermore $V\in L^{\infty}(\Omega)$, then it also holds
\begin{align*}
\| \uLOD - u \| \le C H^4 \qquad \mbox{and} \qquad
| \lambdaLOD - \lambda | \le C H^6.
\end{align*}
In the estimates, $C$ depends on $V$, $\lambda$, $\beta$, $\Omega$, $\| u \|_{L^{\infty}(\Omega)}$ and the mesh regularity of $\mathcal{T}_H$.
\end{theorem}
Compared to the original result proved in \cite{HMP14}, the rates for the $L^2$ and $H^1$-error are both improved by one order each, the rate for the energy by two orders and the rate for the eigenvalue is dramatically improved by three orders. However, we also stress that the equation considered in \cite{HMP14} is slightly more general, involving another (potentially discontinuous) coefficient in the kinetic term. Therefore the rates in \cite{HMP14} are optimal for this generalized setting, but only suboptimal for the GPE. The estimates in Theorem \ref{thm_main} can be even further improved significantly by using a post-processing technique as in \cite{HMP14}.

Finally, we also note that the rates for the $L^2$- and $H^1$-error in Theorem \ref{thm_main} are the same (with respect to space) as for suitable LOD discretizations of the time-dependent Gross--Pitaevskii equation, cf. \cite{HeW21}. Similarly, the energy (which is an invariant of the time-dependent GPE) can be also approximated with a $\mathcal{O}(H^6)$-accuracy if the initial value is sufficiently smooth and if the selected time integrator is energy-conservative.

\section{Proof of Optimal Convergence Rates}\label{sec:superconvergence}
In this section we prove the optimal convergence rates of order $O(H^3)$ in the $H^1$-norm, $O(H^4)$ in the $L^2$-norm, and $O(H^6)$ for both the energy and the eigenvalue. The error bounds presented below all depend on some constant $C$, which may depend on the ground state eigenvalue $\lambda$, the potential $V$, the interaction constant $\beta$, and the domain $\Omega$, but not on the mesh size $H$. 

In the analysis below we shall make use of the $a(\cdot,\cdot)$-orthogonal projection $\PLOD : H^1_0(\Omega) \to \VLOD$ defined by
\begin{align}\label{PLOD_eq}
a(\PLOD(v),w) = a(v,w) \qquad  \mbox{for all } w \in \VLOD.
\end{align}
The following approximation properties of the LOD space are well known and are explicitly stated and proved in \cite[Section 2.1]{HeW21}.
\begin{lemma}\label{PLOD_error}
	Assume (A1)-(A2). Let $v\in H^1_0(\Omega)$ and assume that 
	\begin{align*}
	a(v,w) = (f,w) \quad \mbox{for all } w \in H^1_0(\Omega).
	\end{align*}
	If $f\in L^2(\Omega)$, then 
	\begin{align}\label{PLOD_error_H}
	\|v-\PLOD(v)\| + H\|v-\PLOD(v)\|_{H^1(\Omega)}\leq CH^2\|f\|.
	\end{align}
	If $f\in H^2(\Omega)\cap H^1_0(\Omega)$, then 
	\begin{align}\label{PLOD_error_H3}
	\|v-\PLOD(v)\| + H\|v-\PLOD(v)\|_{H^1(\Omega)}\leq CH^4\|f\|_{H^2(\Omega)}.
	\end{align}
	The constant $C$ only depends on $\Omega$ and the mesh regularity of $\mathcal{T}_H$.
\end{lemma}

\subsection{Convergence in $H^1$-norm}
We start with the $H^1$-error estimate which follows from the abstract $H^1$-best approximation result in Theorem \ref{theorem-approximation-properties-ground-states}, combined with the approximation properties in Lemma \ref{PLOD_error} and the regularity result in Lemma \ref{lemma_bounds_u}. We have:

\begin{proposition}\label{thm_error_H1}
	Let $u$ and $\uLOD$ be solutions to \eqref{weak_problem} and \eqref{LOD_approx}, respectively. It holds that
	\begin{align*}
	\|u - \uLOD\|_{H^1(\Omega)} \leq CH^3.
	\end{align*}
\end{proposition}
\begin{proof}
	We note that 
	\begin{align*}
	a(u,v) = (\lambda u - \beta|u|^2u,v) \qquad \mbox{for all } v \in H^1_0(\Omega),
	\end{align*}
	where $|u|^2u \in H^1_0(\Omega) \cap H^2(\Omega)$ by Lemma \ref{lemma_bounds_u}. Hence, we can apply 
	estimate \eqref{PLOD_error_H3} to get
	\begin{align*}
	\|u-\PLOD(u)\|_{H^1(\Omega)}\leq CH^3\|\lambda u - \beta |u|^2u\|_{H^2(\Omega)}.
	\end{align*}
	From Theorem \ref{theorem-approximation-properties-ground-states} it follows that
	\begin{align*}
	\|u-\uLOD\|_{H^1(\Omega)} \leq C \|u-\PLOD(u)\|_{H^1(\Omega)} \le C H^3\|\lambda u - \beta |u|^2u\|_{H^2(\Omega)}.
	\end{align*}
\end{proof}
\subsection{Convergence in $L^2$-norm}
Next, we turn our attention to proving the $L^2$-error estimate which reads as follows.
\begin{proposition}\label{thm_error_L2}
	Let $u$ and $\uLOD$ be solutions to \eqref{weak_problem} and \eqref{LOD_approx}, respectively. If  $V\in L^{\infty}(\Omega)$, then it holds
	\begin{align*}
	\|u - \uLOD\| \leq C\,H^4.
	\end{align*}
\end{proposition}
The proof is similar to \cite[Section 6.3]{HMP14} with several simplifying modifications that keep the proof short.

\begin{proof}[Proof of Proposition \ref{thm_error_L2}]
Let $e_\LOD=u-\uLOD$. With Theorem \ref{theorem-approximation-properties-ground-states} we have
	\begin{align}
	\label{abstract-L2est-LOD}
	\| e_\LOD \|^2 \leq C\| e_\LOD \|_{H^1(\Omega)} \inf_{\psi \in \VLOD} \|\psi_{e_\LOD} - \psi\|_{H^1(\Omega)},
	\end{align}
	where we recall $\psi_{e_\LOD} \in V^\perp_u = \{v\in H^1_0(\Omega) | \hspace{2pt} (u,v) = 0 \}$ as the solution to the dual problem
\begin{eqnarray}
\label{dual_problem}
\langle (E^{\prime\prime}(u) -\lambda) \psi_{e_\LOD},  v_{\perp} \rangle = ( e_\LOD , v_{\perp} ) \qquad\mbox{for all } v_{\perp} \in V^{\perp}_u.
\end{eqnarray}	
By ellipticity of $E^{\prime\prime}(u) -\lambda$ we have $\| \psi_{e_\LOD} \|_{H^1(\Omega)} \le C \| e_\LOD  \|$. If $P_{V^{\perp}_u}(v):=v - (u,v)u$ denotes the $L^2$-projection of $H^1_0(\Omega)$ onto $V^{\perp}_u$, we can replace the test functions in \eqref{dual_problem} by $v_{\perp}=P_{V^{\perp}_u}(v)$ for arbitrary $v\in H^1_0(\Omega)$. Exploiting 
additionally that $u$ solves the eigenvalue problem and 
the $L^2$-orthogonality induced by the projection, we obtain
\begin{align*}
( P_{V^{\perp}_u}( e_\LOD) , v ) &=  ( e_\LOD , P_{V^{\perp}_u}(v) ) 
\overset{\eqref{dual_problem}}{=} \langle (E^{\prime\prime}(u) -\lambda) \psi_{e_\LOD}, P_{V^{\perp}_u}(v) \rangle \\
&= \langle (E^{\prime\prime}(u) -\lambda) \psi_{e_\LOD}, v \rangle - (u,v) \, \langle (E^{\prime\prime}(u) -\lambda) \psi_{e_\LOD}, u \rangle\\
&\hspace{-3pt}\overset{\eqref{identity-Eprimeprime}}{=} 
 \langle (E^{\prime\prime}(u) -\lambda) \psi_{e_\LOD}, v \rangle - (u,v) \,  2 (\beta|u|^2 u ,  \psi_{e_\LOD} ).
\end{align*}
By rearranging the terms and using $P_{V^{\perp}_u}( e_\LOD) =e_\LOD - (u,e_\LOD)u$ we verify that we 
can express \eqref{dual_problem} as
\begin{align}\label{dual_problem_ext}
\langle (E^{\prime\prime}(u) -\lambda) \psi_{e_\LOD},  v \rangle =
2\beta((|u|^2u,\psi_{e_\LOD})u,v) + (e_\LOD - (e_\LOD,u)u,v)
\end{align}
for all $v \in H^1_0(\Omega)$. Here we note a small typo in the original formulation of \cite{HMP14} where one multiplication with $u$ is missing in the second term in the right hand side. Analogously to the argument in the proof of Lemma \ref{lemma_bounds_u}
we conclude with \eqref{identity-Eprimeprime-prev} that \eqref{dual_problem_ext} can be interpreted as a standard linear elliptic problem of the form
\begin{align*}
- \Delta \psi_{e_\LOD} = (\lambda - V - 3 \beta \, |u|^2 ) \psi_{e_\LOD} + 2\beta (|u|^2u,\psi_{e_\LOD})\, u + e_\LOD - (e_\LOD,u)u =: g
\end{align*}
with right hand side in $g \in L^2(\Omega)$ which we can bound by
\begin{eqnarray*}
\| g \| 
&\le& ( |\lambda| + \|V\|_{L^{\infty}(\Omega)}  + 3 \beta \| u \|_{L^{\infty}(\Omega)}^2 ) \|  \psi_{e_\LOD}  \|\\
&\enspace&\qquad+  2\beta | (|u|^2u,\psi_{e_\LOD})| \, \| u\| + \| e_\LOD \| + | (e_\LOD,u) | \, \|u\| \\
&\overset{\| u\|=1}{\le}& ( |\lambda| + \|V\|_{L^{\infty}(\Omega)}  + 3 \beta \| u \|_{L^{\infty}(\Omega)}^2 + 2 \beta \| u \|_{L^6(\Omega)}^3) \|  \psi_{e_\LOD}  \|
+ 2 \|  e_\LOD  \|.
\end{eqnarray*}
Since we already know that $\|  \psi_{e_\LOD}  \| \le C \|  \psi_{e_\LOD}  \|_{H^1(\Omega)} \le C  \| e_\LOD  \|$ we conclude that
\begin{align*}
\| g \| \le  C  \| e_\LOD  \|.
\end{align*}
With elliptic regularity theory for the Poisson problem $- \Delta \psi_{e_\LOD} = g$ on the convex domain $\Omega$, we conclude that $ \psi_{e_\LOD} \in H^2(\Omega)$ and it holds the regularity estimate $ \| \psi_{e_\LOD} \|_{H^2(\Omega)}  \leq C \| g \|$. Together with the Sobolev embeddings we arrive at
\begin{align}\label{psi_bound}
\| \psi_{e_\LOD} \|_{L^\infty(\Omega)} \le C \| \psi_{e_\LOD} \|_{H^2(\Omega)}  \leq C \| g \| \le  C\| e_\LOD \|.
\end{align}
With estimate \eqref{psi_bound} at hand, we want to estimate $ \inf_{\psi \in \VLOD} \|\psi_{e_\LOD} - \psi\|_{H^1(\Omega)}$ in \eqref{abstract-L2est-LOD}. For that we select $\psi = \PLOD(\psi_{e_\LOD})$ and apply \eqref{PLOD_error_H} and \eqref{dual_problem_ext} to obtain
\begin{align*}
\inf_{\psi \in \VLOD} \|\psi_{e_\LOD} - \psi\|_{H^1(\Omega)} &\leq C H\big(\|3\beta|u|^2\psi_{e_\LOD}\| + \|\lambda \psi_{e_\LOD}\| \\&\qquad +\|2\beta(|u|^2u,\psi_{e_\LOD})u\| + \|e_\LOD - (e_\LOD,u)u\|\big).
\end{align*}
	We analyze the terms on the right hand side one by one. H\"{o}lder's inequality and the Sobolev inequality $\|u\|_{L^6(\Omega)} \leq C\|u\|_{H^1(\Omega)}$ (for $d\leq 3$) give us for the first term
	\begin{align*}
	\||u|^2\psi_{e_\LOD}\| \leq \|u^2\| \, \|\psi_{e_\LOD}\|_{L^\infty(\Omega)} = \|u\|^2_{L^4(\Omega)}\|\psi_{e_\LOD}\|_{L^\infty(\Omega)} \overset{\eqref{psi_bound}}{\leq} C\|u\|^2_{H^1(\Omega)}\|e_\LOD\|.
	\end{align*}
	For the second term we readily have
	$\|\lambda \psi_{e_\LOD}\| \leq C\lambda \|e_\LOD\|$.
	For the third term we use \eqref{der_order_zero} and $\|u\|=1$ to get
	\begin{align*}
	\|2\beta(|u|^2u,\psi_{e_\LOD})u\| &\leq C|(|u|^2u,\psi_{e_\LOD})| \, \|u\|\leq C\||u|^2u\| \, \|\psi_{e_\LOD}\| \\&\leq C\|u\|^3_{H^1(\Omega)}\|e_\LOD\|.
	\end{align*}
	With  $\|u\|=1$, the last term is estimated as
	\begin{align*}
	\|e_\LOD - (e_\LOD,u)u\| &\leq 	2 \|e_\LOD\|.
	\end{align*}
	Combining the estimates for all terms yields
\begin{eqnarray*}
\inf_{\psi \in \VLOD} \|\psi_{e_\LOD} - \psi\|_{H^1(\Omega)} \le C H \|e_\LOD\|.
\end{eqnarray*}	
	Estimate  \eqref{abstract-L2est-LOD} together with the $H^1$-estimate in Proposition~\ref{thm_error_H1} finish the proof.	
\end{proof}

\subsection{Convergence of Energy}
The following error estimate for the energy is a direct consequence of the $H^1$-estimate in Proposition~\ref{thm_error_H1} and the abstract energy estimate \eqref{estimate-nergy-abstract} in Theorem \ref{theorem-approximation-properties-ground-states}.
\begin{proposition}\label{thm_error_energy}
	Let $u$ and $\uLOD$ be solutions to \eqref{weak_problem} and \eqref{LOD_approx}, respectively. It holds that
	\begin{align*}
	| E(u) - E(\uLOD) | \leq C H^6.
	\end{align*} 
\end{proposition}

\subsection{Convergence of Eigenvalues}
The optimal convergence rates for the eigenvalue involve the most substantial changes compared to \cite{HMP14}. In the first step, we prove the following auxiliary result, which established optimal convergence rates in the $H^{-2}$-norm. 
\begin{lemma}\label{lemma_dual_error}
Assume $V\in L^{\infty}(\Omega)$ and let $u$ and $\uLOD$ be the solutions to \eqref{weak_problem} and \eqref{LOD_approx}, respectively. Then it holds
\begin{align*}
\|u-\uLOD\|_{H^{-2}(\Omega)} 
:= \sup_{\omega \in H^2(\Omega)\cap H^1_0(\Omega)\setminus \{0\}} \frac{(\omega,u-\uLOD)}{\|\omega\|_{H^2(\Omega)}}
\leq CH^6.
\end{align*}
\end{lemma}
\begin{proof}
For $\omega \in H^2(\Omega)\cap H^1_0(\Omega)$, let $\psi_\omega \in V^\perp_u$ be the solution to the following dual problem (cf. \eqref{dual_problem_ext}):
\begin{align}\label{dual_problem_ext_omega}
\langle (E^{\prime\prime}(u) -\lambda) \psi_{\omega},  v \rangle =
2\beta((|u|^2u,\psi_{\omega})u,v) + (\omega - (\omega,u)u,v)
\end{align}
for all $v\in H^1_0(\Omega)$. Letting $e_{\LOD}:=u-\uLOD \in H^1_0(\Omega)$ we first observe with $\| u \|=\| \uLOD \|=1$ that 
\begin{align*}
(u,e_{\LOD}) = \frac{1}{2} \left( \| u - e_{\LOD}\|^2 - \| u \|^2  - \| e_{\LOD}\|^2  \right) = -\frac{1}{2}  \| e_{\LOD}\|^2.
\end{align*}
Hence, testing with $v=e_{\LOD}$ in \eqref{dual_problem_ext_omega} and using the above identity we obtain
	\begin{align}\label{main_eq_dualnorm}
	\begin{split}
	(\omega, e_{\LOD}) &= \langle (E''(u)-\lambda)e_{\LOD},P_{\VLOD \cap V^\perp_u}\psi_\omega \rangle \\
	&\quad+ \langle(E''(u)-\lambda)e_{\LOD},\psi_\omega - P_{\VLOD \cap V^\perp_u}\psi_\omega\rangle 
	\\ &\quad+ \| e_{\LOD} \|^2 (\beta|u|^2u,\psi_\omega) - \frac{1}{2}\|e_{\LOD}\|^2 (u,\omega) =: \mbox{I} + \mbox{II} + \mbox{III} + \mbox{IV},
	\end{split}
	\end{align}%
	where $P_{\VLOD \cap V^\perp_u}$ is the $a(\cdot,\cdot)$-orthogonal projection onto $\VLOD \cap V^\perp_u$. Recalling
	\begin{align*}
	\langle (E''(u)-\lambda)v,w \rangle = a(v,w) + 3 \beta(|u|^2v,w) - \lambda \, (v,w)
	\end{align*}
	and using the variational formulations of the eigenvalue problems for $u$ and $\uLOD$ we obtain for the first term
	\begin{align*}
	\mbox{I} &= ((\lambda-\lambdaLOD)\uLOD,P_{\VLOD\cap V^\perp_u}\psi_\omega) \\&\quad+ \beta(2|u|^2(u-\uLOD)+\uLOD(\uLOD^2-u^2),P_{\VLOD\cap V^\perp_u}\psi_\omega) =: \mbox{I}_1 + \mbox{I}_2.
	\end{align*}
	For $\mbox{I}_1$ we use that $(u,P_{\VLOD\cap V^\perp_u}\psi_\omega) = 0$ and $\| P_{\VLOD\cap V^\perp_u}\psi_\omega \|_{H^1(\Omega)} \le C \| \psi_\omega \|_{H^1(\Omega)} \le C \| \omega \|$ (by $H^1$-stability of $P_{\VLOD\cap V^\perp_u}$) so that
	\begin{align*}
	\mbox{I}_1 &= ((\lambda-\lambdaLOD)(\uLOD-u),P_{\VLOD\cap V^\perp_u}\psi_\omega) \leq |\lambda - \lambdaLOD| \, \|u-\uLOD\| \, \| \omega \|.
	\end{align*}
	Proposition \ref{thm_error_L2} and the sub-optimal estimate $|\lambda-\lambdaLOD|\leq CH^3$ from \cite{HMP14} yield $\mbox{I}_1 \le C H^7 \| \omega \|$.
	For $\mbox{I}_2$ we rely on the equality
	\begin{align*}
	2|u|^2(u-\uLOD)+\uLOD(|\uLOD|^2-|u|^2) = (2 u + \uLOD) |u - \uLOD|^2.
	\end{align*}
	Thus, with the embedding $H^1_0(\Omega) \hookrightarrow L^6(\Omega)$ we have
	\begin{align*}
	\mbox{I}_2 &\leq C((2u + \uLOD) |u - \uLOD|^2,P_{\VLOD\cap V^\perp_u}\psi_\omega) \\
	&\leq C\|2u + \uLOD\|_{L^6(\Omega)} \, \| |u-\uLOD|^2\|_{L^3(\Omega)}\| \, \|P_{\VLOD\cap V^\perp_u}\psi_\omega\|_{L^6(\Omega)} \\
	&\leq C \|u-\uLOD\|^2_{H^1(\Omega)}\|\psi_\omega\|_{H^1(\Omega)} \leq CH^6\|\omega\|,
	\end{align*}
	where we have used that 
	$\| 2u + \uLOD\|_{L^6(\Omega)} \le C (\sqrt{\lambda} + \sqrt{\lambda_{\LOD}}) \le C (\sqrt{\lambda} + H^{3/2})$.
	We deduce $\mbox{I} \le CH^6\|\omega\| \le C H^6 \|\omega\|_{H^2(\Omega)}$.
	
	For the second term in \eqref{main_eq_dualnorm} we deduce the bound
	\begin{align*}
	\mbox{II} &\leq C\|u-\uLOD\|_{H^1(\Omega)}\|\psi_\omega - P_{\VLOD\cap V^\perp_u}\psi_\omega\|_{H^1(\Omega)} \\&\leq CH^3\|\psi_\omega - P_{\VLOD\cap V^\perp_u}\psi_\omega\|_{H^1(\Omega)}.
	\end{align*}
	To estimate the right hand side, we first note that (cf. \cite{CCM10,HMP14})
	$$
	\| \psi_{\omega} - P_{\VLOD\cap V^\perp_u}\psi_\omega \|_{H^1(\Omega)}
	\le C \| \psi_{\omega} - P_{\LOD} \psi_\omega \|_{H^1(\Omega)}
	$$
	by using the $H^1$-stability properties of  $P_{\VLOD\cap V^\perp_u}$ and $P_{\LOD}$ together with the fact that $ \psi_{\omega} \in V_u^{\perp}$ (and assuming that $H$ is small enough). Hence, we can apply  Lemma~\ref{PLOD_error} by writing $\psi_\omega$ (with the same argument as in \eqref{dual_problem_ext}) as
	\begin{align*}
	a(\psi_\omega,v) = (\omega - (\omega,u)u,v) + (2\beta(|u|^2u,\psi_\omega)u,v) + (\lambda\psi_\omega,v) - 3\beta(|u|^2\psi_\omega,v).
	\end{align*}
	We conclude with \eqref{PLOD_error_H3} that
	\begin{align*}
	\|\psi_\omega - P_{\VLOD\cap V^\perp_u}\psi_\omega\|_{H^1(\Omega)} &\leq CH^3 \big(\|\omega\|_{H^2(\Omega)} + \|(\omega,u)u\|_{H^2(\Omega)} \\&\quad+ \|(|u|^2u,\psi_\omega)u\|_{H^2(\Omega)} + \|\lambda\psi_\omega\|_{H^2(\Omega)} + \||u|^2\psi_\omega\|_{H^2(\Omega)} \big) \\&=: CH^3\sum_{i=1}^5 \mbox{II}_i
	\end{align*}
	For $\mbox{II}_2$ we use Lemma~\ref{lemma_bounds_u}
	\begin{align*}
	\mbox{II}_2 = |(\omega,u)|\|u\|_{H^2(\Omega)} \leq \|\omega\| \, \|u\| \, \|u\|_{H^2(\Omega)} \leq C\|\omega\|_{H^2(\Omega)}.
	\end{align*}
	For the last terms we use Lemma~\ref{lemma_bounds_u} again together with the regularity bound $\|\psi_\omega\|_{H^2(\Omega)}\leq C\|\omega\|$ (which is obtained analogously as in the proof of Proposition \ref{thm_error_L2}) to conclude
	\begin{align*}
	\mbox{II}_3 + \mbox{II}_4 + \mbox{II}_5 &\leq \||u|^2u\|\|\psi_\omega\|\,\|u\|_{H^2(\Omega)} + \lambda\|\psi_\omega\|_{H^2(\Omega)} + \|u\|_{L^\infty(\Omega)}^2\|\psi_\omega\|_{H^2(\Omega)} \\&\leq C\|\omega\| \leq C\|\omega\|_{H^2(\Omega)}.
	\end{align*}
	We deduce $\mbox{II} \leq CH^6\|\omega\|_{H^2(\Omega)}$.
	
	For the third and fourth term in \eqref{main_eq_dualnorm} we use the error bound in the $L^2$-norm from Proposition~\ref{thm_error_L2}, the regularity of $\psi_\omega$, and Lemma~\ref{lemma_bounds_u} to derive
	\begin{align*}
	\mbox{III} + \mbox{IV} \leq CH^8\||u|^2u\| \, \|\psi_\omega\| + CH^8\|\omega\|_{H^1(\Omega)} \leq CH^8\|\omega\|_{H^2(\Omega)}.
	\end{align*}	
	Combining everything, we obtain $(\omega,u-\uLOD) \leq CH^6\|\omega\|_{H^2(\Omega)}$, which gives
	\begin{align*}
	\|u-\uLOD\|_{H^{-2}} = \sup_{\omega \in H^2\cap H^1_0\setminus \{0\}} \frac{(\omega,u-\uLOD)}{\|\omega\|_{H^2(\Omega)}} \leq CH^6.
	\end{align*}
\end{proof}
With this we are ready to prove the final estimate for the eigenvalue, which also finishes the proof of Theorem \ref{thm_main}.
\begin{proposition}\label{thm_error_eigen}
Assume $V\in L^{\infty}(\Omega)$. Let $\lambda$ and $\lambdaLOD$ be the eigenvalues corresponding to $u$ and $\uLOD$ in \eqref{groundstate-EV-from-energy} and \eqref{LOD_approx} respectively. It holds that
	\begin{align*}
	|\lambda - \lambdaLOD| \leq CH^6.
	\end{align*}
\end{proposition}
\begin{proof}
	We use the definition of the eigenvalues to estimate
	\begin{align*}
	\|\lambda - \lambdaLOD\| \leq 2|E(u) - E(\uLOD)| + \frac{\beta}{2}|\|u\|^4_{L^4(\Omega)} - \|\uLOD\|^4_{L^4(\Omega)}|,
	\end{align*}
	where the first term is bounded by Proposition~\ref{thm_error_energy}. It remains to bound the second term. We have
	\begin{eqnarray*}
	\lefteqn{ \| \uLOD \|^4_{L^4(\Omega)} - \| u \|^4_{L^4(\Omega)} = ((|u|^2 + |\uLOD|^2)(\uLOD+u),\uLOD-u) }
	\\&=& ((|u|^2 + |\uLOD|^2), |\uLOD-u|^2) + (2u(|u|^2 + |\uLOD|^2),\uLOD-u)
	\\&=& \mbox{I} + \mbox{II}.
	\end{eqnarray*}
	For the first term we use Proposition~\ref{thm_error_H1} to achieve
	\begin{align*}
	\mbox{I} &\leq \| \, |u|^2 + |\uLOD|^2\| \, \|(\uLOD-u)^2\| \leq C\|\uLOD-u\|^2_{L^4(\Omega)} \leq C\|\uLOD-u\|^2_{H^1(\Omega)} \\&\leq CH^6.
	\end{align*}
	The second term splits into
	\begin{align*}
	\mbox{II} &= (2u(|u|^2 + |\uLOD|^2),\uLOD-u) \\
	&= 4(|u|^2u, \uLOD-u) + 2((\uLOD+u)u,|\uLOD-u|^2) \\
        &\leq C(\|u^2u\|_{H^2(\Omega)}\|u-\uLOD\|_{H^{-2}} + \|(\uLOD+u)u\| \, \|u-\uLOD\|^2_{L^4(\Omega)})
	\\&\leq C(\|u^2u\|_{H^2(\Omega)}\|u-\uLOD\|_{H^{-2}} + \|(\uLOD+u)u\| \, \|u-\uLOD\|^2_{H^1(\Omega)}).
	\end{align*}
	We note that by H\"{o}lder's inequality and Sobolev embeddings we have the bound
	\begin{align*}
	\|(\uLOD+u)u\| &\leq \|\uLOD\|_{L^4(\Omega)}\|u\|_{L^4(\Omega)}  + \|u\|_{L^4(\Omega)}^2 \\
	&\leq C \left( \|\uLOD\|_{H^1(\Omega)}\|u\|_{H^1(\Omega)} + \|u\|_{H^1(\Omega)}^2 \right) \leq C.
	\end{align*}
	This relies on the fact that $\|\uLOD\|_{H^1(\Omega)} \leq C $, which can simply be deduced by using the error bound in the $H^1$-norm
	\begin{align*}
	\|\uLOD\|_{H^1(\Omega)}\leq \|\uLOD - u\|_{H^1(\Omega)}  + \|u\|_{H^1(\Omega)} \leq CH^3 + \|u\|_{H^1(\Omega)} \leq C(1+H^3).
	\end{align*}
	Thus, using Lemma~\ref{lemma_dual_error} and Proposition~\ref{thm_error_H1} we deduce $\mbox{II}\leq CH^6$, which completes the proof.
\end{proof}

\section{Numerical Experiments}\label{sec:experiments}
In this section we perform two numerical experiments to validate the theoretical convergence rates for both a smooth and a discontinuous potential $V$. The numerical experiments are performed for the ideal method, without localization of the corrections, to eliminate any potential reduction in convergence rate coming from the size of the localization patches. We emphasize that these experiments are only to verify that the derived rates are indeed optimal for the Gross--Pitaevskii equation and in a practical application the localization should always be implemented. For further details regarding localization we refer to \cite{HMP14} where the method was first suggested. 

\subsection{Smooth potential}\label{subsec:smooth_potential}
For the smooth potential we consider a setup similar to \cite{BD04}. The computational domain is given by the square $\Omega=[-6,6]\times[-6,6]$, the potential is set to the smooth (harmonic) function $V(x,y) = 0.5(x^2 + y^2)$, and the interaction parameter to $\beta=100$. 

The resulting minimizing problem \eqref{LOD_approx} is solved by using the iterative scheme, referred to as normalized gradient flow, presented in \cite{BD04}.

A reference solution is computed on a fine mesh of size $h = 2^{-4}$. The LOD approximations $\uLOD$ are computed for meshes of decreasing size $H=2^{1},2^{0},2^{-1},2^{-2}$ and then compared to the reference solution. The resulting convergence rates for the $H^1$-norm, the $L^2$-norm, the energy, and the eigenvalue are plotted in Figure~\ref{fig:smooth_V}. The plots clearly confirm the convergence orders predicted by Theorem \ref{thm_main}.

\begin{figure}
	\centering
	\begin{subfigure}[b]{0.5\textwidth}
		\includegraphics[width=\textwidth]{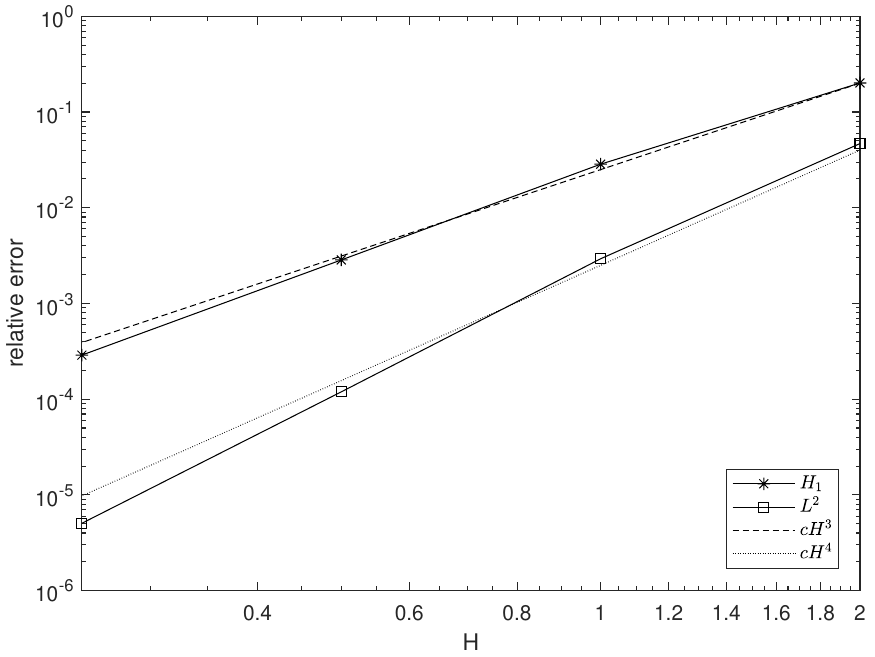}
		\caption{$H^1$- and $L^2$-norm}
	\end{subfigure}~
	\begin{subfigure}[b]{0.5\textwidth}
		\includegraphics[width=\textwidth]{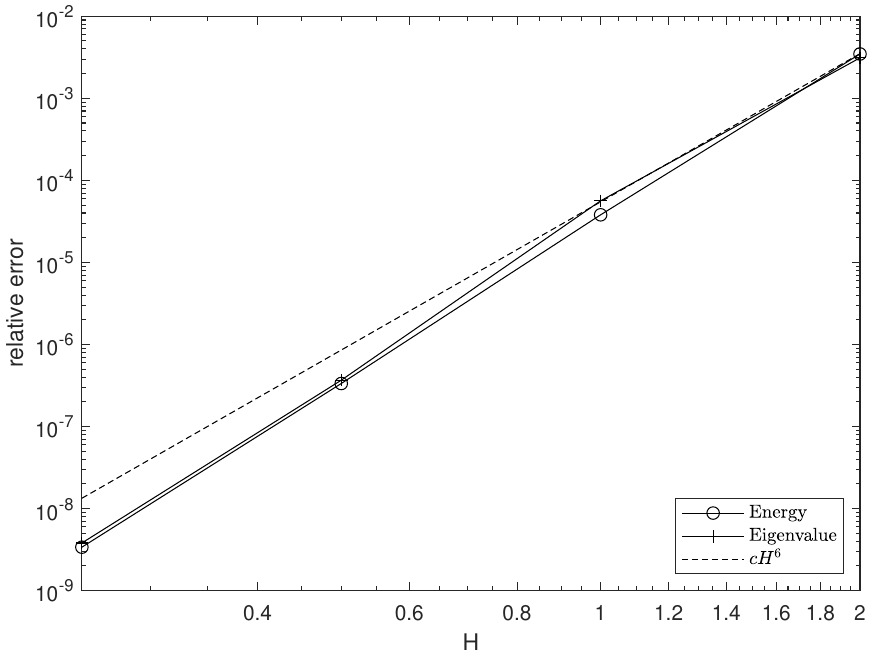}
		\caption{Energy and eigenvalue}
	\end{subfigure}
	\caption{Relative errors for $\uLOD$, $E(\uLOD)$, and $\lambdaLOD$ for the smooth potential used in the numerical experiment in Section \ref{subsec:smooth_potential}.} \label{fig:smooth_V}
\end{figure}

\begin{figure}
	\centering
	\begin{subfigure}[b]{0.45\textwidth}
		\includegraphics[width=\textwidth]{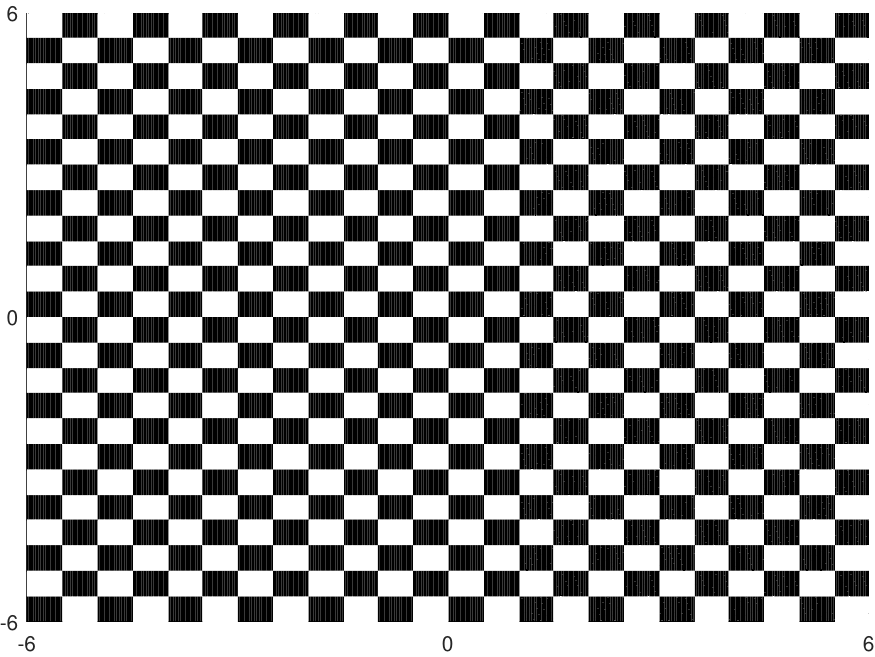}
		\caption{A realization of the checkerboard potential.}
	\end{subfigure}~~
	\begin{subfigure}[b]{0.5\textwidth}
		\includegraphics[width=\textwidth]{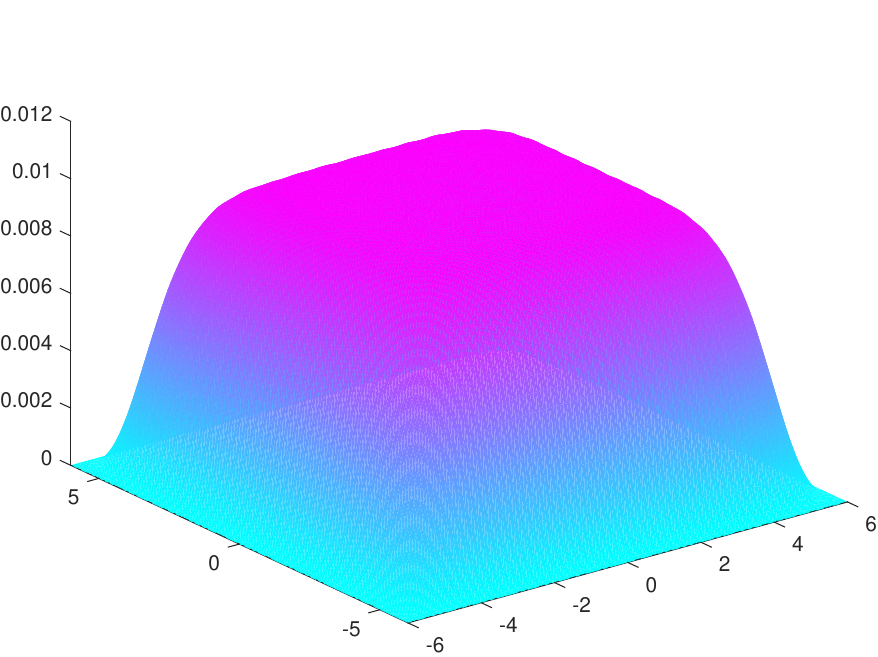}
		\caption{The corresponding ground state $|u|^2$ computed on the reference mesh.}
	\end{subfigure}
	\caption{A plot of the checkerboard potential and the corresponding ground state for the numerical experiment in Section \ref{subsec:disc_potential}.} 
	\label{fig:checkerboard_and_groundstate}
\end{figure}

\begin{figure}
	\centering
	\begin{subfigure}[b]{0.5\textwidth}
		\includegraphics[width=\textwidth]{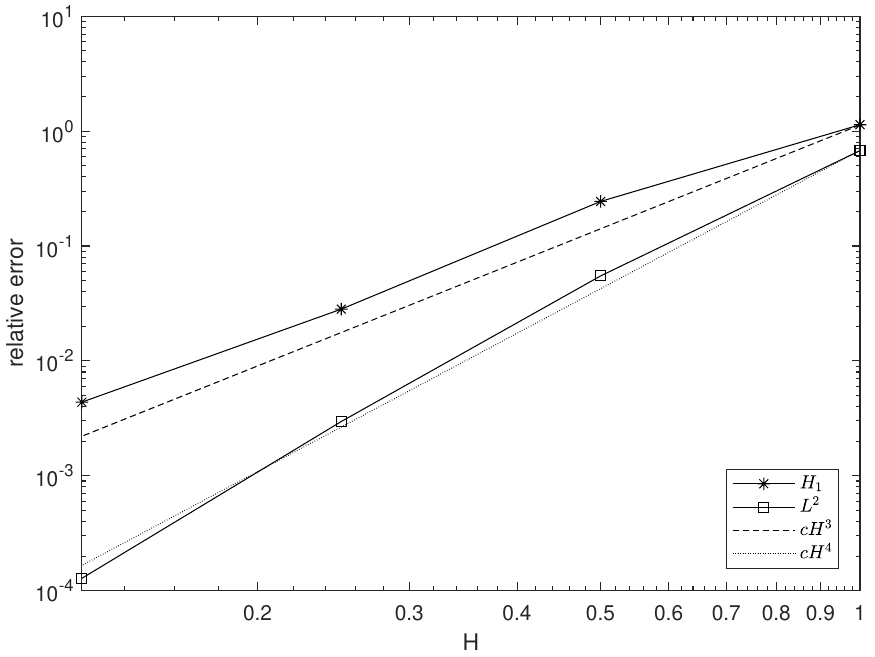}
		\caption{$H^1$- and $L^2$-norm}
	\end{subfigure}~
	\begin{subfigure}[b]{0.5\textwidth}
		\includegraphics[width=\textwidth]{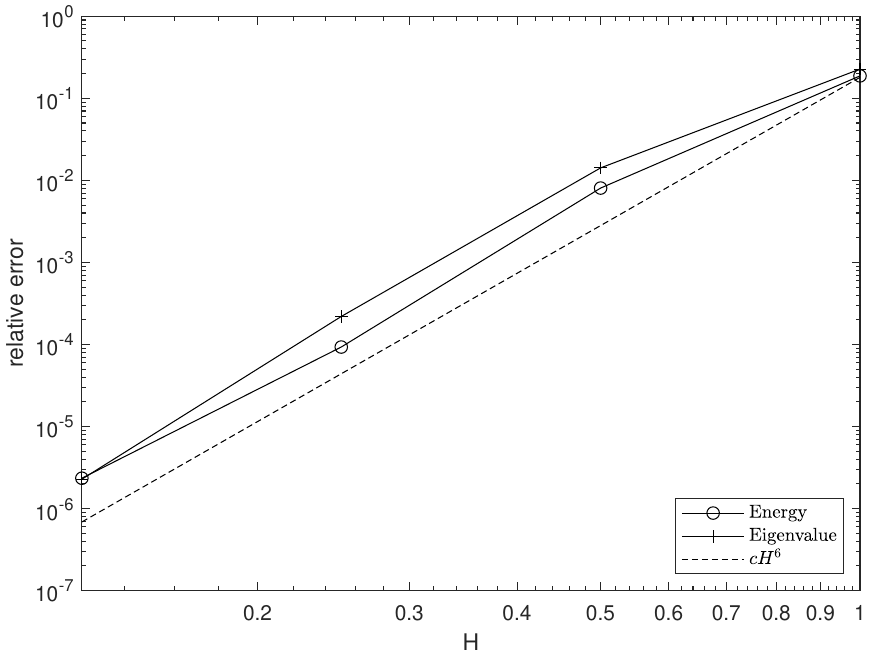}
		\caption{Energy and eigenvalue}
	\end{subfigure}
	\caption{Relative errors for $\uLOD$, $E(\uLOD)$, and $\lambdaLOD$ for the discontinuous potential used in Section \ref{subsec:disc_potential}.} \label{fig:disc_V}
\end{figure}

\subsection{Discontinuous potential}\label{subsec:disc_potential}

For the discontinuous potential the setting remains roughly the same as for the smooth potential in Section~\ref{subsec:smooth_potential} with $\Omega=[-6,6]\times[-6,6]$ and $\beta=100$. 

The potential is set to a checkerboard with squares of size $2^{-2}$ which results in $24\times 24$ squares on the domain. The values at the squares alternate between $0$ and $1$, see Figure~\ref{fig:checkerboard_and_groundstate}.

A reference solution is computed on a mesh of size $h = 2^{-5}$. The triangles in the mesh are aligned with the discontinuities in the potential, so that the potential can be computed exactly on each triangle to avoid additional numerical errors. The LOD approximations $\uLOD$ are computed for meshes of decreasing size $H=2^{0},2^{-1},2^{-2},2^{-3}$ and then compared to the reference solution. The resulting convergence rates for the $H^1$-norm, the $L^2$-norm, the energy, and the eigenvalue are plotted in Figure~\ref{fig:disc_V}. Again, the plots clearly confirm the predicted convergence orders.

\medskip
$\\$
{\bf Acknowledgements.}
The authors would like to thank the anonymous reviewers for their insightful comments that helped to improve the paper.

\end{document}